\documentclass[leqno,11pt,letterpaper, english]{amsart}
\usepackage[usenames,dvipsnames]{color}
\usepackage[colorlinks=true,linkcolor=Red,citecolor=Green]{hyperref}
\usepackage{amsmath,inputenc,euscript,amssymb,geometry}
\usepackage{graphicx}
\usepackage{amssymb}
\usepackage{latexsym}
\usepackage{amssymb,amsbsy,amsmath,amsfonts,amssymb,amscd,color}
\usepackage{mathrsfs}
\usepackage{dsfont}

\newcommand{\xqedhere}[1]{%
    \rlap{%
         \hbox to#1{%
           \hfil
           \llap{%
               \ensuremath{\square}
           }%
       }%
   }%
}

\def\pasdegrille{\let\grille = \pasgrille}

\def\aat#1#2#3{
\divide \dimen1 by 48 \dimen3=\dimen1 \multiply \dimen1 by #1
\advance \dimen1 by -\dimen3 \divide \dimen1 by 101 \multiply
\dimen1 by 100 \divide \dimen2 by \count11 \multiply \dimen2 by #2
\setbox0=\hbox{#3}\ht0=0pt\dp0=0pt
  \rlap{\kern\dimen1 \vbox to0pt{\kern-\dimen2\box0\vss}}\dimen1= \wd1
\dimen2=\ht1}
\def\pasgrille{
\count12= \dimen1 \divide \count12 by 50 \divide \dimen2 by
\count12 \count11 =\dimen2 \ \divide \dimen1 by 48
\setlength{\unitlength}{\dimen1} \smash{\rlap{\ }} \dimen1= \wd1
\dimen2=\ht1 }
\def\grille{
\count12= \dimen1 \divide \count12 by 50 \divide \dimen2 by
\count12 \count11 =\dimen2 \ \divide \dimen1 by 48
\setlength{\unitlength}{\dimen1}
\smash{\rlap{\graphpaper[1](0,0)(50, \count11)}} \dimen1= \wd1
\dimen2=\ht1 }

\pasdegrille
\newcommand{\ud}{\,\mathrm{d}}

\newcommand{\be}{\begin{equation}}
\newcommand{\ee}{\end{equation}}

\newcommand{\ra}{\rangle}
\newcommand{\la}{\langle}

\newcommand{\CC}{{\mathbb C}}
\newcommand{\NN}{{\mathbb N}}

\newcommand{\RR}{{\mathbb R}}

\newcommand{\vol}{\operatorname{vol}}

\renewcommand{\Re}{\mathop{\rm Re}\nolimits}
\renewcommand{\Im}{\mathop{\rm Im}\nolimits}

\theoremstyle{plain}

\newtheorem{thm}{Theorem}

\newtheorem{prop}{Proposition}[section]
\newtheorem{cor}[prop]{Corollary}
\newtheorem{lem}[prop]{Lemma}

\theoremstyle{definition}

\numberwithin{equation}{section}

\def\squarebox#1{\hbox to #1{\hfill\vbox to #1{\vfill}}} 
\newcommand{\norm}[1]{\Vert#1\Vert}

\usepackage{amsxtra}

\ifx\pdfoutput\undefined
  \DeclareGraphicsExtensions{.pstex, .eps}
\else
  \ifx\pdfoutput\relax
    \DeclareGraphicsExtensions{.pstex, .eps}
  \else
    \ifnum\pdfoutput>0
      \DeclareGraphicsExtensions{.pdf}
    \else
      \DeclareGraphicsExtensions{.pstex, .eps}
    \fi
  \fi
\fi

\title[Damped wave equation]{A remark on the logarithmic decay of the damped wave and Schr{\"o}dinger equations on a compact Riemannian manifold}

\author[N. Burq]{Nicolas Burq}
\address{Universit{\'e} Paris Sud/ Université Paris-Saclay, Math{\'e}matiques, B{\^a}t 307, 91405  Orsay Cedex, France,  UMR 8628 du CNRS and Institut Universitaire de France}
\email{Nicolas.burq@math.u-psud.fr}
\author[I. Moyano]{Iv\'an Moyano}
\address{Laboratoire Jean-Alexandre Dieudonné, Universit\'e C\^ote-d'Azur, 28 Parc Valrose, 06028 Nice cedex, France }
\email{imoyano@unice.fr}

\usepackage{amssymb}
\usepackage{amsmath, amsthm, amsopn, amsfonts}

\begin{document}    

\begin{abstract}
In this paper we consider a compact Riemannian manifold $(M,g)$ of class $C^1 \cap W^{2,\infty}$ and the damped wave or Schr\"odinger equations on $M$, under the action of a damping function $a = a(x)$. We establish the following fact: if the measure of the set $ \left\{ x \in M; \, a(x) \not = 0 \right\} $ is strictly positive, then the decay in time of the associated energy is at least logarithmic.

\end{abstract}   

\maketitle   

\section{Introduction}  
\label{sec:Introduction}

Consider a compact Riemannian manifold $(M,g)$ of class $C^1 \cap W^{2,\infty}$, possibly with boundaries $\partial M$, endowed with a Lipschitz metric $g$. Denote $\ud_g x$ (or simply $dx$) the volume element in $M$ associated to the metric $g\in C^0 \cap W^{1, \infty}$ and write $\vol_g$ the associated volume on $M$. Let $\Delta$ be the Laplace-Beltrami operator in $(M,g)$. Recall that in local coordinates we may write
\be
\Delta = \frac{1}{\sqrt{\det g}} \partial_i \left( \sqrt{\det g} g^{ij} \partial_j \right).
\label{eq:LaplaceBeltrami}
\ee In this note we are interested in the evolution of respectively the wave equation and the Schr\"odinger equation under the influence of a damping term localised via  a function $0\leq a(x), a \in L^\infty (M) $ non trivial ($\int_M a(x) dx >0$) which consequently may be supported  in a small subset of~$M$, namely $E$. We shall briefly recall these models. \par

The damped wave equation in $M$ under the damping $a\partial_t u$ corresponds to the initial value problem   
\begin{equation}
\left\{ \begin{array}{ll}
\partial_t^2 u - \Delta u + a(x) \partial_t u = 0, & \RR_+ \times M, \\
(u, \partial_t u) |_{t=0} = (u_0,u_1), & M,
\end{array} 
\right.
\label{eq:waves}
\end{equation} where $(u_0,u_1)$ is a given initial condition in the natural energy space $\mathcal{H} = H^1(M) \times L^2(M)$.  If $\partial M \not = \emptyset$, we impose the boundary conditions
\begin{equation}
u \mid_{\partial M} =0 \text{ (Dirichlet condition) } \quad \textrm{or} \quad   \partial_\nu u \mid_{\partial M} =0 \text{ (Neumann condition)}.
\label{eq:waves BC}
\end{equation} The energy associated to (\ref{eq:waves}) is as usual
\begin{equation}
\mathcal{E}_{w}(t,u_0,u_1) = \int_M |\partial_t u(t)|^2 \ud x + \int_M |\nabla_x u(t)|^2 \ud x, \qquad t \geq 0,
\label{eq:energy} 
\end{equation} defined globally as  $u \in C^0(\RR_+;H^1(M)) \cap C^1(\RR_+;L^2(M))$. \par

The second model we are interested in, is the initial value problem for the Schr\"odinger equation under the action of the damping $a = a(x)$, i.e.,
\begin{equation}
\left\{ \begin{array}{ll}
i \partial_t \psi + \Delta \psi + i a(x) \psi = 0, & \RR_+ \times M, \\
\psi |_{t=0} = \psi_0, & M,
\end{array} 
\right.
\label{eq:schrodinger}
\end{equation} for a given $\psi_0 \in L^2(M;\CC)$. Again if $\partial M \not = \emptyset$, we impose the boundary conditions (\ref{eq:waves BC}). The energy associated to (\ref{eq:schrodinger}) is 
\begin{equation}
\mathcal{E}_S(t,\psi_0) = \int_M |\psi(t)|^2 \ud x.
\label{eq:energy Schrodinger} 
\end{equation} In this note we prove that if $E \subset M$ is any measurable set with $\vol_g(E) > 0$, the energy functionals $\mathcal{E}_w$ and $\mathcal{E}_S$ decay at least logarithmically in time. This is the content of Theorems \ref{thm:decay} and \ref{thm:decay schrodinger} below.

\subsection{Main results}
Since we assume $a\geq 0, \int_M a(x) dx >0$, we deduce that there exists $n>0$ such that the set 
$$ F_n = \{ x \in M; a(x) > \frac 1 n\}$$ 
has positive measure. As a consequence, with $\alpha = \frac 1 n, \beta = \| a\|_{L^\infty}$, $F = F_n$, 
We get that the damping function $a = a(x)$ satisfies 
\begin{equation}
\alpha 1_F(x) \leq a(x) \leq \beta, \qquad \text{ for almost all } x \in M,
\label{eq:hypothesis damping}
\end{equation} with $F \subset M$ of positive measure, not necessarily open. Our main result for the wave equation is the following.

\begin{thm}
Let $a\geq 0, \int_M a(x) dx >0$. Then,  there exists a constant $C=C(F)>0$ such that for every 
\begin{equation}
(u_0,u_1) \in \begin{cases} \left(H^{2}(M)\cap H^1_0(M) \right) \times H^1_0(M), \qquad  \, \text{ with Dirichlet boundary conditions},\\
H^{2}(M) \times H^1(M), \qquad \qquad  \qquad \qquad   \text{otherwise},
\end{cases}
\end{equation} the solution to the associated damped wave equation (\ref{eq:waves}) satisfies
 \begin{equation}
\mathcal{E}_w(t,u_0,u_1) \leq \frac{C}{\log(2 + t)^{2}} \left( \| u_0 \|_{H^{2}(M)}^2  +  \| u_1 \|_{H^{1}(M)}^2 \right), \qquad t \geq 0,
\label{eq:decay log}
\end{equation} where $\mathcal{E}_w$ is the energy defined in (\ref{eq:energy}).
\label{thm:decay}
\end{thm} See
Section \ref{sec:Notation} for notation. In the case of the Schr\"odinger equation we obtain the next analogous result.

\begin{thm}\label{energy decay Schrodinger}
Let $a\geq 0, \int_M a(x) dx >0$. Then, there exists a constant $C=C(F)>0$ such that for every 
 \begin{equation}
\psi_0 \in \begin{cases} H^{2}(M)\cap H^1_0(M), \qquad  \text{with Dirichlet boundary conditions},\\
H^{2}(M), \qquad  \qquad  \qquad     \text{ otherwise},
\end{cases}
\end{equation}the solution to the associated Schr\"odinger equation (\ref{eq:schrodinger}) satisfies

 \begin{equation}
\mathcal{E}_S(t,\psi_0) \leq \frac{C}{\log(2 + t)^{4}}  \| \psi_0 \|_{H^{2}(M)}^2, \qquad t \geq 0,
\label{eq:decay log schrodinger}
\end{equation} where $\mathcal{E}_S$ is the energy defined in (\ref{eq:energy Schrodinger}).
\label{thm:decay schrodinger}
\end{thm}

The strategy of proof combines the spectral inequalities obtained in \cite{BM} (see Theorem \ref{thm:spectral inequality} below) with a sharp characterisation of the logarithmic decay of energy from \cite{BurqActa,BattyDuyckaerts} (see Theorem \ref{thm:Burq} in Section \ref{sec:A criterion for logarithmic decay}). We give some details concerning these results in Section \ref{sec:Some tools}.

\subsection{Notation and setting}
\label{sec:Notation}

As mentioned above, the volume induced by $g$ on $M$ is defined as 
\begin{equation*}
\vol_g(A) = \int_M 1_A(x) \ud_g x, 
\end{equation*} for every Borel set $A \subset M$. In the case of the Euclidean flat space $\RR^d$, me simply write $|A|$ for the $d-$dimensional Lebesgue measure of a given Borel set $A \subset \RR^d$. In both cases we denote $B(x,r)$ the ball of radius $r>0$ centred at a point $x$.  

As $M$ is compact, the Laplace-Beltrami operator on $M$ defined in (\ref{eq:LaplaceBeltrami}) has compact resolvent. Let $(e_k)_{k\in \NN}$ be the family of $L^2$-normalised eigenfunctions of $- \Delta$, with eigenvalues $\lambda_k^2 \rightarrow + \infty$ and satisfying 
$$ -\Delta e_k = \lambda_k ^2 e_k, \qquad e_k \mid_{\partial M} =0 \text{ (Dirichlet condition) or } \partial_\nu e_k \mid_{\partial M} =0 \text{ (Neumann condition)}.$$

Recall that $(e_k)_{k\in \NN}$ is a Hilbert basis of $L^2(M)$ endowed with the usual inner product $\la \cdot,  \cdot \ra$. Moreover, the usual Sobolev norms on $M$ can be defined using the spectral basis $(e_k)_{k\in \NN}$ as follows:
\be
\| f \|_{H^s(M)}^2 = \sum_{k \in \NN} (1 + \lambda_k^2) |\la f, e_k \ra|^2, \qquad f \in H^s(M), 
\ee for every $s \in \RR$.

\subsection{Previous work}

\subsubsection{Decay of damped waves} The study of the decay rates for (\ref{eq:waves}) has been addressed by the seminal works \cite{BLR} and \cite{Lebeau96}. These work establish an intimate relation between the rate decay of the energy and the support of the damping function. Under the geometric control condition of \cite{BLR} one can expect an exponential decay, as shown for instance in \cite{BLR,Lebeau96,BurqJoly}. On the other hand, when the support of $a$ does not satisfy a geometric control condition, the decay rate of the associated damped wave equation may be slower than exponential. We can find examples in the literature of polynomial decay \cite{AnantharamanLeautaud,Phung} or even logarithmic \cite{Lebeau96,LeRo95, BurqJoly, LaurentLeautaud}. \par 

Under some hypothesis on the geometry of the manifold, such as the assumption that the manifold is compact and hyperbolic (negative curvature), it is possible to expect exponential decay in some (positive) Sobolev spaces as soon as $a$ is smooth and non zero (cf. \cite{Jin}). \par 

In this paper we establish the following fact: if $| \left\{ x \in M; \, a(x) \not = 0 \right\}| > 0 $, then the decay is at least logarithmic. We do not make any assumption on the curvature of the manifold.

\subsubsection{Spectral inequalities}
\label{sec:specral inequalities}

In the framework described in Section \ref{sec:Notation}, Given a small subset $E\subset M$ (of positive Lebesgue measure or at least not too small), we have studied in \cite{BM} how $L^p$ norms of the restrictions to $E$ of arbitrary finite linear combinations of the form
$$\phi= \sum_{\lambda_k \leq \Lambda} u_k e_k (x)$$
can dominate Sobolev norms of $\phi$ on the whole $M$. Our result \cite[Thm 2]{BM} is the following.

\begin{thm}\label{spectral}
Let $(M,g)$ be a Riemannian manifold of class $C^1 \cap W^{2,\infty}$, possibly with boundaries $\partial M$. There exists $\delta\in (0,1)$ such that for any $m>0$, there exist $C, D>0$ such that for any $\omega \subset M$ with $ \vol_g(\omega) \geq m$ and for any $\Lambda >0$,   
we have
\begin{equation}\label{borne}
\phi= \sum_{\lambda_k \leq \Lambda} u_k e_k (x) \Rightarrow \| \phi\|_{L^2(M)} \leq C e^{D \Lambda} \| \phi 1_{\omega}\|_{L^2(M)}.
\end{equation}
\label{thm:spectral inequality}
\end{thm}

We shall use the spectral inequality (\ref{borne}) in Section \ref{sec:Estimates for the Helmholtz equation}.

\subsection{Outline}

In Section \ref{sec:Some tools} we gather some facts about the tools used in the proof of our main result: 
Section \ref{sec:A criterion for logarithmic decay} is devoted to a characterisation of logarithmic decay,
Section \ref{sec:Estimates for the Helmholtz equation} makes the link between the resolvent operators for waves and Schr\"odinger and the Helmholtz equation, and 
Section \ref{sec:Estimates for the Helmholtz equation} is concerned with some estimates for solutions to the Helmholtz equation obtained thanks to the spectral inequalities mentioned before. 
The proof of Theorem~\ref{thm:decay} is carried out in Sections \ref{sec:Proof of Theorem} and \ref{sec:Neumann} which treat respectively the case of Dirichlet boundary conditions and the case of Neumann boundary conditions. Each of these two sections is also divided into high frequencies and low frequencies.
Finally, Section \ref{sec:Proof of Theorem schrodinger} is concerned with the proof of Theorem \ref{thm:decay schrodinger}, which is also divided into Section \ref{sec:negative frequencies} (negative frequencies) and Section \ref{sec:positive frequencies} (nonnegative frequencies).

\section{Some tools}
\label{sec:Some tools}

In this section we describe first (in Section \ref{sec:A criterion for logarithmic decay}) some abstract results relating the time decay of a semi-group with the growth of the resolvent operator at infinity. Next, in Section \ref{sec:Resolvent and the Helmholtz equation} we focus on the resolvent operators related to the wave equation (\ref{eq:waves}) and the Schr\"odinger equation (\ref{eq:schrodinger}), which in both cases lead to a Helmholtz equation of the form
\begin{equation*}
\Delta u + \lambda u = S, 
\end{equation*} for some parameter $\lambda \in \RR$ and a source term $S$. Finally in Section \ref{sec:Estimates for the Helmholtz equation} we use Theorem \ref{thm:spectral inequality} to get some estimates for solutions to the Helmholtz equation that will be useful in the sequel.

\subsection{Sufficient conditions for logarithmic decay}
\label{sec:A criterion for logarithmic decay}

Consider a Hilbert space $\mathcal{H}$ and the functional equation 
\begin{equation}
\frac{\ud U }{\ud t }  = A U, \quad t \geq 0, \qquad \textrm{with} \qquad U(0) = U_0 \in \mathcal{H},
\label{eq:abstract equation}
\end{equation} for a possibly unbounded operator $A$ with domain $D(A) \subset \mathcal{H}$. As usual, $z \in \CC$ belongs to the resolvent set $\rho(A)$ whenever $(A - z)^{-1} \in \mathcal{L}(\mathcal{H})$. The spectrum of $A$ is $\sigma(A) = \CC \setminus \rho(A)$.

We focus next on the elements of $\rho(A)$ lying on the imaginary axis. For every $\tau \in \RR$ we consider the resolvent mapping $R(\tau) = (A - i \tau)^{-1}$ whenever $i\tau \in \rho(A)$.

\subsubsection{Growth of the resolvent and decay of the semi-group}

Assume that $A$ is the infinitesimal generator of a $C^0$-continuous semi-group of operators in $\mathcal{H}$ that we denote $(e^{tA})_{t \geq 0}$, so that the solution to (\ref{eq:abstract equation}) writes $U(t) = e^{tA} U_0$. Assume further that
\begin{equation*}
\sup_{t \geq 0}  \norm{ e^{tA} }_{\mathcal{L}(\mathcal{H})}  < + \infty. 
\end{equation*}

Batty and Duyckaerts have introduced in \cite{BattyDuyckaerts} a quantitative approach to characterise the asymptotic behaviour of the semi-group, i.e., the fact that for some $k \in \NN^*$,
\begin{equation}
\lim_{t \rightarrow + \infty}  m_k(t)  = 0, \qquad \textrm{with} \qquad m_k(t) = \norm{ e^{tA} (Id - A)^{-k} }_{\mathcal{L}(\mathcal{H})}, \quad t \geq 0,
\label{eq:abstract semigroup decay}
\end{equation} in terms of the purely spectral condition 
\begin{equation}
\sigma(A) \cap i \RR = \emptyset. 
\label{eq:spectral condition}
\end{equation} Observe that this condition ensures that the resolvent operators $R(\tau)$ are well defined for any $\tau \in \RR$. Moreover, it is possible to describe the decay rate of (\ref{eq:abstract semigroup decay}) in terms of the growth of the function 
\begin{equation}
M(\mu) = \sup_{\vert \tau \vert \leq \mu} \norm{ (A - i\tau )^{-1} }_{\mathcal{L}(\mathcal{H})}, \qquad  \mu \in [0, + \infty).
\label{eq:function M}
\end{equation}

The following result, obtained in \cite[Thm. 3]{BurqActa}, guarantees logarithmic decay of all $m_k$ as long as $M$ grows at most exponentially at infinity. 

\begin{thm}[Lebeau-Robbiano~\cite{LeRo95}, Burq \cite{BurqActa}, Batty-Duyckaerts, \protect{\cite[Thm 1.5]{BattyDuyckaerts}}] \label{thm}Assume that~\eqref{eq:spectral condition} holds and that 
\begin{equation*}
\exists C,c > 0 \quad \textrm{such that} \quad M(\mu) < C e^{c| \mu |}, (\text{resp. } M(\mu) < C e^{c\sqrt{| \mu |}} )\qquad \forall \mu \in \RR.
\end{equation*} Then, for any $k > 0$ there exists $C_k$ such that 
\begin{equation*}
\left\|  e^{tA} (Id - A)^{-k}  \right\|_{\mathcal{L}(\mathcal{H})} \leq \frac{C_k}{\log(2 + t)^k}, \qquad \forall t \geq 0.
\end{equation*}

\begin{equation*}
(\text{resp.} \left\|  e^{tA} (Id - A)^{-k}  \right\|_{\mathcal{L}(\mathcal{H})} \leq \frac{C_k}{\log(2 + t)^{2k}}, \qquad \forall t \geq 0.)
\end{equation*}
\label{thm:Burq}
\end{thm}

We shall use this result in Sections \ref{sec:Proof of Theorem} and \ref{sec:Proof of Theorem schrodinger} to get the decay of solutions to (\ref{eq:waves}) and (\ref{eq:schrodinger}).

\subsection{The resolvent operator and the Helmholtz equation}
\label{sec:Resolvent and the Helmholtz equation}

In this section we make explicit the choice of functional framework (compatible with Section \ref{sec:A criterion for logarithmic decay}) associated to the wave equation and the Schr\"odinger equations.

\subsubsection{The resolvent operator for waves}\label{sec:resolvent operator for waves} Following the notation of Section \ref{sec:A criterion for logarithmic decay}, let us set 
\begin{equation}
A U = \begin{pmatrix}
0 & Id \\ \Delta & -a(x) 
\end{pmatrix}, \qquad U = \begin{pmatrix}
u \\ v
\end{pmatrix},
\label{eq:operator waves}
\end{equation} in the Hilbert space $\mathcal{H} = H^1(M) \times L^2 (M)$ endowed with the natural inner product. As usual, $D(A) = H^2(M) \times H^1(M)$ if $\partial M = \emptyset$ or $D(A) = ( H^2(M) \cap H^1_0(M)) \times H^1(M)$ if we impose Dirichlet boundary conditions (the case of Neumann boundary conditions is slightly more involved and we deal with it in Section \ref{sec:Neumann}). The solution of (\ref{eq:abstract equation}) is given by 
\begin{equation}
U(t) = e^{tA} \begin{pmatrix}
u_0 \\ u_1
\end{pmatrix}
\label{eq:semigroup waves}
\end{equation} and the solution of (\ref{eq:waves}) is given by the first component of $U(t)$. Let $\tau \in \RR$ and consider the resolvent operator $R(\tau) = (A - i \tau )^{-1}$. For any $\begin{pmatrix}
f \\ g
\end{pmatrix} \in \mathcal{H}$, one has 
\begin{equation*}
\begin{pmatrix}
u \\ v 
\end{pmatrix} = (A - i \tau)^{-1} \begin{pmatrix}
f \\ g
\end{pmatrix} \Leftrightarrow   \left\{  \begin{array}{l}
v - i \tau u = f,  \\ \Delta u  - a v - i \tau v = g.
\end{array}     \right.
\end{equation*} Using that $v = i\tau u + f$, we find
\begin{equation*}
\Delta u  - a  (i \tau u + f ) + \tau^2 u - i \tau f = g, \qquad \textrm{in } M, 
\end{equation*} and hence $u$ satisfies the Helmholtz equation
\begin{equation}
\Delta u  +  \tau^2 u = g  + ( a + i \tau ) f  + i a  \tau u, \qquad \textrm{in } M.
\label{eq:Helmholtz2}
\end{equation}

\subsubsection{The resolvent operator for Schr\"odinger}
\label{sec: The resolvent operator for Schrodinger} In this case we set $\mathcal{H} = L^2(M;\CC)$ and
\be
A = i \Delta - a(x), \qquad D(A) = H^2(M;\CC). 
\label{eq:operator Schrodinger}
\ee For $\psi_0 \in \mathcal{H}$ given, the solution of (\ref{eq:schrodinger}) is then given by the 
\be
U(t)\psi_0 = e^{tA} \psi_0, \qquad  t \geq 0.  
\label{eq:semigrooup Schrodinger}
\ee Now, let $\tau \in \RR$ and consider the resolvent operator $R(\tau) = (A - i \tau)^{-1}$. If for some $f \in \mathcal{H}$ the function $\psi \in D(A)$ is such that $(A - i \tau)^{-1}\psi = f$, then $\psi$ satisfies the following Helmholtz equation: 
\be
\Delta \psi - \tau \psi + i a(x) \psi = -i f, \qquad \textrm{in } M.
\label{eq:Helmholtz Schrodinger}
\ee

\subsection{Estimates for the Helmholtz equation}
\label{sec:Estimates for the Helmholtz equation}

We state for convenience a unique continuation result for the Helmholtz equation that will be useful in Section \ref{sec: Proof Waves Low}. The unique continuation from \emph{small} sets follows from the Remez inequalities obtained in \cite[Section 1 Eq. (6)  ]{LMpropSmallness}. 

\begin{lem}
Let $\omega \subset M$ be a measurable set with $\vol_g(\omega)>0$. Let $\lambda \in \RR$ be fixed and let $u$ be the solution to the Helmholtz equation 
\begin{equation}
\Delta u + \lambda u = 0, \qquad \textrm{in } M.
\label{eq:Helmholtz uc}
\end{equation} Then, $u$ satisfies the unique continuation principle on $\omega$, i.e.,
\begin{equation}
u|_{\omega} = 0 \qquad \Rightarrow \qquad  u|_{M} = 0.
\label{eq:Helmholtz uc}
\end{equation}
\label{lemma:Helmholtz}
\end{lem}

We state next an estimate for solutions to the Helmholtz equation with source that will be used in Section \ref{sec: Proof Waves High} to get a suitable exponential bound in the high frequency regime. The following result follows from the spectral inequality (\ref{borne}) in Theorem \ref{thm:spectral inequality} above.

\begin{prop}
Let $\omega \subset M$ be a measurable set with $\vol_g(\omega)>0$. There exist constants $C= C(\omega)> 0$ and $D =D(\omega) > 0$ such that for every $\mu \in \RR$ and $S \in L^2(M)$, the solution to the Helmholtz equation 
\begin{equation}
\Delta u + \mu^2 u = S, \qquad \textrm{in } M, \qquad u\mid_{\partial M} =0 \text{ (Dirichlet) or } \partial_\nu u \mid_{\partial M} =0 \text{ (Neumann)} 
\label{eq:Helmholtz}
\end{equation} satisfies
\begin{equation}
\| u \|_{L^2(M)} \leq Ce^{D|\mu|} \left(  \|S \|_{L^2(M)} + \| 1_{\omega} u \|_{L^2(M)} \right).
\label{eq:Helmholtz estimate}
\end{equation}
\label{prop:Helmholtz}
\end{prop}

\begin{proof}

Let $u$ be given by (\ref{eq:Helmholtz}). If $S \in L^2(M)$, using the orthonormal basis $(e_k)$ as in Section \ref{sec:Notation}, we can write 
\begin{equation*}
S = \sum_{k \in \NN} S_k e_k.
\end{equation*} Then, we can split $u$ into ``hyperbolic" and ``elliptic"  frequencies as follows
\begin{equation*}
u = 1_{| \Delta + \mu^2| \leq 1 }u + 1_{| \Delta + \mu^2| > 1 }u,
\end{equation*} where 
\begin{equation*}
1_{| \Delta + \mu^2| > 1 }u = \sum_{k \in \NN; \, \mu^2 - k^2> 1} \frac{S_k}{\mu^2 - k^2} e_k.
\end{equation*} Thanks to this explicit expression, we have 
\begin{equation}
\left\|  1_{| \Delta + \mu^2| > 1 }u   \right\|_{L^2(M)} \leq \left\|  S   \right\|_{L^2(M)}.
\label{eq:estimate high frequencies Helmholtz}
\end{equation} Now applying Theorem \ref{spectral} on the ``hyperbolic" frequencies, we get 
\begin{align*}
\left\|  1_{| \Delta + \mu^2| \leq 1 }u   \right\|_{L^2(M)} & \leq Ce^{D|\mu|} \left\|  1_{| \Delta + \mu^2| \leq 1 }u   \right\|_{L^2(\omega)} \\
& \leq  Ce^{D|\mu|} \left(  \left\|  1_{| \Delta + \mu^2| \leq 1 }u   \right\|_{L^2(\omega)} +  \left\|  1_{| \Delta + \mu^2| > 1 }u   \right\|_{L^2(\omega)}       \right) \\
& \leq  Ce^{D|\mu|} \left(  \left\|  1_{| \Delta + \mu^2| \leq 1 }u   \right\|_{L^2(\omega)} +  \left\|  S   \right\|_{L^2(M)}       \right) \\
& \leq  Ce^{D|\mu|} \left(  \left\|  u   \right\|_{L^2(\omega)} +  \left\|  S  \right\|_{L^2(M)}       \right),
\end{align*} where we have used (\ref{eq:estimate high frequencies Helmholtz}).
\end{proof}

\section{Proof of Theorem \ref{thm:decay} for Dirichlet boundary conditions}
\label{sec:Proof of Theorem}

Following the notation of Section \ref{sec:A criterion for logarithmic decay}, let $\tau \in \RR$ and consider the resolvent operator $R(\tau) = (A - i \tau )^{-1}$. Throughout this section we assume that the boundary conditions are of Dirichlet type only and use the notation of Section \ref{sec:resolvent operator for waves}.
% Define $\tau_{*} > 0$ to be a large enough real number such that (\ref{eq:choice of tau}) holds.  Next we separate two regimes : The high frequencies $\tau \in \RR$ such that $|\tau| \geq \tau^*$ and the low frequencies $\tau \in \RR$ such that $|\tau_*| < \tau^*$. 

In Sections \ref{sec: Proof Waves High} we prove the resolvent estimate for the wave equations for $\tau \geq \tau^*$ (with $\tau^*$ a constant sufficiently large. Then in Section \ref{sec: Proof Waves Low} we prove the estimate for $\tau \leq \tau^*$.

\subsection{Proof of Theorem \ref{thm:decay}: High frequencies}
\label{sec: Proof Waves High}

\begin{prop} Let $F \subset E \subset M$ and a damping $a$ satisfying (\ref{eq:hypothesis damping}). Then there exist $\tau_* \geq 1$ large enough and constants  $C_h, c_h>0$ independent of $\tau$ such that for every  $\vert \tau \vert \geq \tau_*$ we have 
\be
\| U \|_{H^1(M) \times L^2(M)} \leq C_h e^{c_h |\tau| } \| (f,g ) \|_{H^1(M) \times L^2(M)},  
\label{eq:estimate high frequencies}
\ee for every $h = (f,g) \in \mathcal{H}$ and every $U =(u,v) = (A - i \tau )^{-1} h \in D(A)$. 
\label{prop: high freq}
\end{prop}

\begin{proof} 
First recall that for any $ h = \begin{pmatrix}
f \\ g
\end{pmatrix} \in \mathcal{H}$, the element $U= \begin{pmatrix}
u \\ v
\end{pmatrix}= (A - i \tau )^{-1} h$ satisfies the Helmholtz equation (\ref{eq:Helmholtz2}) with the boundary conditions 
$$ u\mid_{\partial M}= 0$$ and the identities 
\begin{equation*}
v - i \tau u = f  \quad \textrm{and} \quad   \Delta u  - a v - i \tau v = g, \qquad \textrm{in } M.
\end{equation*} The first equation yields 
\begin{equation*}
\norm{v}_{L^2} \leq  |\tau| \norm{u}_{L^2} + \norm{ f}_{L^2} 
\end{equation*} and the second one, after multiplying by $\overline{u}$ and integrating, gives, for any $\vert \tau \vert \geq 1$, 

\begin{align*}
\norm{\nabla u }_{L^2}^2  & \leq  \int_M  \left| \left( (a + i \tau) v + g \right)\overline{u} \right| \ud x \\
& =  \int_M \left| \left( (a + i \tau) (f + i \tau u) + g \right)\overline{u} \right| \ud x \\
& \lesssim ( 1 + \beta + \tau^2) \norm{u}_{L^2} \left( \norm{f}_{L^2} + \norm{g}_{L^2} + \norm{u}_{L^2}\right),
\end{align*} where we have used that $a \leq \beta$. We deduce
\begin{equation*}
\norm{ u }_{H^1} \leq C ( 1 + \beta + \tau^2) \left( \norm{f}_{L^2} + \norm{g}_{L^2} + \norm{u}_{L^2}\right). 
\end{equation*} Hence,  it is sufficient to estimate $\norm{u }_{L^2}$ to get an estimate on $\| U \|_{H^1(M) \times L^2(M)}$. Let us focus on  $\norm{u }_{L^2}$. Recalling that $u$ satisfies the Helmholtz equation (\ref{eq:Helmholtz2}), using Proposition \ref{prop:Helmholtz} with 
\begin{equation*}
\omega = F \qquad \textrm{and} \qquad S = g  + ( a + i \tau ) f  + i a  \tau u, \qquad \mu = \tau,
\end{equation*} the estimate (\ref{eq:Helmholtz estimate}) yields
\begin{align*}
\| u \|_{L^2(M)} & \leq Ce^{D|\tau|} \left( \|  g  + ( a + i \tau ) f  + i a  \tau u \|_{L^2(M)}  + \| 1_F u \|_{L^2(M)}     \right)  \\
& \leq Ce^{D|\tau|} \left( \|  g  \|_{L^2(M)}  + ( |\tau| + \beta) \|  f  \|_{L^2(M)}    +  |\tau| \| a u \|_{L^2(M)}  +  \| 1_F u \|_{L^2(M)}   \right)  \\
& \leq (1 + |\tau| + \beta) Ce^{D|\tau|} \| (f,g ) \|_{\mathcal{H}} + Ce^{D|\tau|}  \left( |\tau|  \| a u \|_{L^2(M)}  +  \| 1_F u \|_{L^2(M)} \right). 
\end{align*} On the other hand, (\ref{eq:hypothesis damping}) implies 
\begin{equation*}
\| 1_F u \|_{L^2(M)} \leq \frac{\sqrt{\beta}}{\alpha} \| \sqrt{a} u \|_{L^2(M)},
\end{equation*} and 
\begin{equation*}
\| a u \|_{L^2(M)} \leq \sqrt{\beta}  \| \sqrt{a} u \|_{L^2(M)}.
\end{equation*} As a result, we get 
\begin{equation}
\| u \|_{L^2(M)} \leq (1 + |\tau| + \beta) Ce^{D|\tau|} \| (f,g ) \|_{\mathcal{H}} + \left(  |\tau| +  \frac{1}{\alpha}  \right) \sqrt{\beta} Ce^{D|\tau|} \| \sqrt{a} u \|_{L^2(M)}.
\label{eq:ProofMain FirstEstimate}
\end{equation}
Next, we need to estimate $\| \sqrt{a} u \|_{L^2(M)}$ in terms of  $\| (f,g) \|_{\mathcal{H}}$ and $\tau$. Using (\ref{eq:Helmholtz2}) we obtain
\begin{equation*}
\int_M (\Delta u  +  \tau^2 u ) \overline{u} \ud x  =  \int_M \left(  g  + ( i \tau + a) f  + i a  \tau u \right) \overline{u} \ud x
\end{equation*} and hence,
\begin{equation*}
-\int_M |\nabla_x u|^2 \ud x + \tau^2 \int_M  |u |^2 \ud x  =  \int_M \left(  g  + ( i \tau + a) f \right) \overline{u} \ud x  + i \tau \int_M a  |u|^2 \ud x.
\end{equation*} Taking the imaginary part, we find 
\begin{align*}
| \tau | \left\| \sqrt{a} u  \right\|^2_{L^2(M)}  & =  \left| \Im \int_M \left(  g  + ( i \tau + a) f \right) \overline{u} \ud x \right| \\
 & \leq \frac{1}{2\epsilon} \left\|  g  + ( i \tau + a) f \right\|_{L^2(M)}^2 + \frac{\epsilon}{2} \| u \|_{L^2(M)}^2 \\
 & \leq \frac{1}{2\epsilon} \left\|  g \right\|_{L^2(M)}^2  + \frac{1}{2 \epsilon} ( | \tau | + \beta^2 )  \left\|  f \right\|_{L^2(M)}^2 + \frac{\epsilon}{2} \| u \|_{L^2(M)}^2,
\end{align*} for every $\epsilon>0$. Choosing 
\begin{equation*}
\epsilon = \frac{  | \tau | e^{-2 D|\tau|}  }{2   \left(  | \tau | +  \frac{1}{\alpha}  \right)^2 \beta C^2  }
\end{equation*} one finds
\begin{align*}
\left\| \sqrt{a} u  \right\|^2_{L^2(M)} & 
\leq \frac{1}{\tau^2} \left(  | \tau | +  \frac{1}{\alpha}  \right)^2 \beta C^2  e^{2D|\tau|}  \left( \left\|  g \right\|_{L^2(M)}^2  + ( | \tau | + \beta^2 )  \left\|  f \right\|_{L^2(M)}^2  \right)  \\
& \qquad \qquad \qquad \qquad  + \frac{   e^{-2 D|\tau|}  }{4   \left(  | \tau | +  \frac{1}{\alpha}  \right)^2 \beta C^2  } \| u \|_{L^2(M)}^2 
\end{align*} and thus,
\begin{align*}
\left\| \sqrt{a} u  \right\|_{L^2(M)} & 
\leq \left(  1 +  \frac{1}{\alpha |\tau|}  \right) \sqrt{\beta} C \left( 1 + \sqrt{ |\tau| + \beta^2 }  \right)  e^{D|\tau|}  \|(f,g) \|_{\mathcal{H}}  \\
& \qquad \qquad \qquad \qquad  + \frac{   e^{-D|\tau|}  }{2   \left(  |\tau| +  \frac{1}{\alpha}  \right) \sqrt{\beta} C  } \| u \|_{L^2(M)}. 
\end{align*} Now, we get from (\ref{eq:ProofMain FirstEstimate})
\begin{multline*}
\| u \|_{L^2(M)}  \leq (1 + | \tau | + \beta) Ce^{D|\tau|} \| (f,g ) \|_{\mathcal{H}} + \left(  | \tau | +  \frac{1}{\alpha}  \right) \sqrt{\beta} Ce^{D|\tau|} \| \sqrt{a} u \|_{L^2(M)} \\
 \leq (1 + | \tau | + \beta) Ce^{D|\tau|} \| (f,g ) \|_{\mathcal{H}} \hfil \\
  + \frac{1}{| \tau |} \left(  | \tau | +  \frac{1}{\alpha}  \right)^2 \beta C^2 e^{2D|\tau|}  \left( 1   + \sqrt{ | \tau | + \beta^2 } \right) \| (f,g ) \|_{\mathcal{H}} + \frac{1}{2} \| u \|_{L^2(M)}.
\end{multline*} Then,  if  $\tau_*$ is large enough so that $\tau^* \geq 1$, we have 
\begin{align}
\frac{1}{2} \| u \|_{L^2(M)} & \leq C \left( (1 + | \tau | + \beta)   +    e^{D|\tau|} C \beta (1  +  \sqrt{ | \tau | + \beta^2 } ) \left(  | \tau | +  \frac{1}{\alpha}  \right)^2  \right)  e^{D|\tau|} \| (f,g ) \|_{\mathcal{H}}  \label{eq:choice of tau} \\
 & \leq  C' e^{2 D|\tau|} \| (f,g ) \|_{\mathcal{H}} ,  \nonumber
\end{align} for every $\vert \tau \vert \geq \tau_*$, where $C'$ is a positive constant depending only on $\alpha,\beta,C,D,\tau^*$. Hence, estimate (\ref{eq:estimate high frequencies}) follows for some constants $C_h,c_h$ large enough, depending only on $\alpha,\beta,C,D,$ and $\tau_*$.

\end{proof}

\subsection{Proof of Theorem \ref{thm:decay}: Low frequencies  (Dirichlet boundary conditions)}
\label{sec: Proof Waves Low}

\begin{prop} Let $F \subset E \subset M$ and a damping $a$ satisfying (\ref{eq:hypothesis damping}). For any $\tau_*>0$ (we shall choose $\tau_*$ given by Proposition \ref{prop: high freq}), there exist a constant $C_{\ell} >0$ such that for any $|\tau| < \tau_*$, every $h = (f,g) \in \mathcal{H}$ and every $U$ such that $U = (A - i \tau)^{-1} h$. 
\be
\| U \|_{H^1(M) \times L^2(M)} \leq C_{\ell} \| (f,g ) \|_{H^1(M) \times L^2(M)}.
\label{eq:estimate low frequencies}
\ee 
\label{prop: low freq} 
\end{prop}

\begin{proof} In this case, we proceed by contradiction (we follow \cite[Sect. 4]{BurqJoly}). Assume that the exponential growth is not true, i.e., that there exist sequences $(U_n) \subset \mathcal{H}$ and $(\tau_n) \subset \mathbb{R}$ with $|\tau_n| \leq \tau_*$ such that 
\begin{equation}
\| U_n \|_{H^1(M) \times L^2(M)} = 1, \quad \forall n \in \NN, \qquad \textrm{and} \qquad (A - i \tau_n) U_n \rightarrow 0 \textrm{ as } n \rightarrow + \infty. 
\label{eq:lowfreq hypothesis}
\end{equation} Writing $U_n = \begin{pmatrix} u_n \\ v_n  \end{pmatrix}$, we have 
\be
v_n - i\tau_n u_n \rightarrow 0, \quad \textrm{in } H^1(M) \qquad \textrm{and} \qquad \Delta u_n - i\tau_n a(x) u_n + \tau_n ^2 u_n \rightarrow 0 \quad \textrm{in } L^2(M).
\label{eq:consequence low freq}
\ee Now multiplying the last limit by $\overline{u_n}$ and integrating by parts we find 
\be
-\| \nabla u_n \|_{L^2(M)}^2 - i \tau_n \int_M a(x) |u_n|^2 \ud x + \tau_n ^2 \| u_n \|_{L^2(M)}^2 \rightarrow 0. \nonumber
\ee Taking real and imaginary parts yields
\begin{equation}
-\| \nabla u_n \|_{L^2(M)}^2  + \tau_n^2 \| u_n \|_{L^2(M)}^2 \rightarrow 0, \qquad \textrm{and} \qquad \tau_n \int_M a(x) |u_n|^2 \ud x \rightarrow 0. 
\label{eq:consequence 2 low freq}
\end{equation} 
The sequence $(\tau_n)$ is bounded (in modulus) by $\tau_*$ and consequently  we can assume that it converges to some limit $\tau$. 
We distinguish now two cases. \par

\textbf{Case $\tau  = 0$.} In this case, we would have 
\be
\| \nabla u_n \|_{L^2(M)}^2 \rightarrow 0,  
\nonumber
\ee thanks to (\ref{eq:consequence 2 low freq}). Hence, by Poincaré's inequality, we would also have  
\be
\| u_n \|_{L^2(M)}^2 \rightarrow 0.  
\nonumber
\ee But then, the first part of (\ref{eq:consequence low freq}) would also imply 
\be
\| v_n \|_{L^2(M)}^2 \rightarrow 0.  
\nonumber
\ee Henceforth,
\be
\| U_n \|_{H^1(M)}^2 \rightarrow 0,  
\nonumber
\ee which is a contradiction with (\ref{eq:lowfreq hypothesis}).

\par
\textbf{Case $\tau \not = 0$.} In this case, using (\ref{eq:consequence 2 low freq}) we may write 
\be
\lim_{n \rightarrow \infty} \| \nabla u_n \|_{L^2(M)}^2 = \lim_{n \rightarrow \infty} \tau^2 \| u_n \|_{L^2(M)}^2, \text{ and } \lim_{n\rightarrow \infty} \int_M a(x) |u_n|^2(x) \ud x =0.
\nonumber
\ee Using (\ref{eq:lowfreq hypothesis}) we also have 
\be
\lim_{n \rightarrow \infty} \| v_n \|_{L^2(M)}^2 = \lim_{n \rightarrow \infty} \tau^2 \| u_n \|_{L^2(M)}^2.
\nonumber
\ee Then, as  
\begin{equation} \label{estim}
1= \lim_{n \rightarrow \infty} \| U_n \|_{H^1(M) \times L^2(M)}^2 = \lim_{n \rightarrow \infty} (1 + 2\tau^2) \| u_n \|_{L^2(M)}^2,
\end{equation} which means that the sequence $(u_n)$ is bounded in $H^1$. 
Then, Rellich's compactness theorem implies that there exists $u \in H^1(M)$ such that
\be
u_n \rightarrow u \quad \textrm{in }L^2(M)-strong, \qquad \textrm{and} \qquad \nabla u_n \rightarrow \nabla u \quad \textrm{in }L^2(M)-weak, 
\nonumber
\ee and a fortiori, upon extracting a subsequence, we also have
\be
u_n \rightarrow u \quad \textrm{a.e. } M. 
\nonumber
\ee Now, thanks to Fatou's lemma and the second part of (\ref{eq:consequence 2 low freq}) we deduce
\be
\int_M a(x) |u|^2 \ud x \leq \liminf_{n \rightarrow \infty} \int_M a(x) |u_n|^2 \ud x = 0 
\nonumber
\ee and hence, using (\ref{eq:hypothesis damping}),
\begin{equation*}
u = 0  \qquad \textrm{a.e. in } F. 
\end{equation*} This is enough to apply Lemma \ref{lemma:Helmholtz} which yields 
\be
u = 0  \qquad \textrm{a.e. in } M. 
\nonumber 
\ee But this is a contradiction with ~\eqref{estim}. This concludes the proof.

\end{proof}

\subsection{End of the proof of Theorem \ref{thm:decay} (Dirichlet boundary conditions)}

\label{sec:proof Dirichlet}

Once we have dealt with high and low frequencies in the previous sections, the proof of Theorem \ref{thm:decay} is a consequence of Theorem \ref{thm:Burq}.

\begin{proof}[Proof of Theorem \ref{thm:decay}]
Combining Proposition \ref{prop: high freq} and Proposition \ref{prop: low freq}, we find that the estimate 
\begin{equation*}
\| U \|_{H^1 \times L^2} \leq  C  e^{2 c|\tau|}  \| h \|_{H^1 \times L^2}
\end{equation*} holds for every $\tau \in \RR$, every $h \in H^1(M) \times L^2(M)$, $U = R(\tau)h$ and the constants
\begin{equation*}
C = C_{\ell} +  C_{h}, \qquad  c = c_{\ell}.
\end{equation*} As a consequence the function $M$ defined by (\ref{eq:function M}) satisfies in this case the growth
\begin{equation*}
M(\mu)  \leq C e^{2 c|\mu|},   
\end{equation*} for every $\mu \in \RR$. Hence, Theorem \ref{thm:Burq} yields that for any $k \in \NN$, 
\begin{equation}
\left\|  U(t) (Id - A)^{-k}  \right\|_{\mathcal{L}(H^1 \times L^2)} \leq \frac{C_k}{\log(2 + t)^k}, \qquad \forall t \geq 0.
\label{eq:decay waves en norme}
\end{equation} where we have used the notation of (\ref{eq:operator waves}) and (\ref{eq:semigroup waves}). Next, let $k \in \NN$ be fixed and let $h_0 = \begin{pmatrix}  u_0 \\ u_1    \end{pmatrix} \in H^{k+1}(M) \times H^{k}(M)$ Then, $W_0 := (Id - A)^{k} h_0 \in L^2(M)$ and inequality (\ref{eq:decay waves en norme}) implies 
\begin{align*}
\mathcal{E}_w(t,u_0, u_1) & = \left\|  U(t) h_0  \right\|^2_{H^1 \times L^2} \\
& = \left\|  U(t) (Id - A)^{-k} h_0  \right\|^2_{H^1 \times L^2} \\
& \leq \frac{C_k}{\log(2 + t)^{2k}}  \left\| W_0 \right\|^2_{H^1 \times L^2} \\
& \leq \frac{C_k}{\log(2 + t)^{2k}}  \left\| (Id - A)^{k} h_0 \right\|^2_{H^1 \times L^2},
\end{align*} for every $t \geq 0$. The proof is completed by noticing that for $k=1$ the domain of $A$ is $(H^2(M)\cap H^1_0 (M)) \times H^1_0 (M)$ with Dirichlet boundary conditions (resp. $H^2 (M) \times H^1(M)$ if $\partial M = \emptyset$), and 
$$\| (Id - A) h_0 \|^2_{H^1 \times L^2} \leq C \|  h_0 \|^2_{H^2 \times H^1}.$$

\end{proof}

\section{Proof of Theorem \ref{thm:decay}: the case of Neumann boundary conditions} \label{sec:Neumann}
In the case $\tau=0$, we cannot use Poincaré's inequality (that we used when dealing with the low frequencies in Proposition \ref{prop: low freq}) and we have to change slightly the functional framework. Here we follow the exposition in~\cite[Appendix]{BuGe20}. 
For the sake of completeness, we recall the argument (which is taken from~\cite[Appendix]{BuGe20}) below and focus on the low frequency regime $|\tau| \leq \tau^*$.

 For $s= 1,2$,  we define $\dot{H}^s = H^s(M) / \mathbb{R}$ the quotient space of $H^s(M) $ by the constant functions, endowed with the norm 
 $$ \| \dot{u} \| _{\dot{H}^1} = \| \nabla u\|_{L^2}, \qquad \| \dot{u} \| _{\dot{H}^2} = \| \Delta u\|_{L^2},$$
 (here $\dot{u}$ denotes the equivalence class of a function $u$ in $\dot{H}$).
 We define the operator 
  $$ \dot{A}= \begin{pmatrix} 0 & {\Pi} \\ \dot{\Delta} & -a \end{pmatrix}$$ 
  on $\dot{H}^1 \times L^2$ with domain $\dot{H}^2 \times H^1$, where $\Pi$ is the canonical projection $H^1 \rightarrow \dot{H}^1$ and $\dot{\Delta}$ is defined by 
  $$ \dot{\Delta} \dot{u} = \dot{( \Delta u)}$$ (independent of the choice of $u\in \dot{u}$). The operator $\dot{A}$ is maximal dissipative and hence defines a semi-group of contractions on $\dot{\mathcal{H}} = \dot{H^1} \times L^2$. Indeed for $U= \begin{pmatrix}\dot{u}\\ v\end{pmatrix}$, 
  $$ \Re \bigl( \dot{A} U, U\bigr)_{\dot{\mathcal{H}}} = \Re ( \nabla u, \nabla v)_{L^2} + ( \dot{\Delta}\dot{u} -a v, v)_{L^2} = - ( av,v) _{L^2},$$
  and 
  \begin{equation}
\begin{aligned} (\dot{A} -\text{Id} ) \begin{pmatrix}\dot{u}\\ v\end{pmatrix} = \begin{pmatrix}\dot f\\ g\end{pmatrix} &\Leftrightarrow \Pi v - \dot u = \dot f, \quad \textrm{and} \quad  \dot{\Delta} \dot{u} -(a+1) v = g\\
 & \Leftrightarrow \Pi v - \dot u = \dot f, \quad \textrm{and} \quad \Delta v - (1+a) v = g + \Delta f \in H^{-1} (M)
  \end{aligned}
  \end{equation}
  and we can solve this equation by variational theory. Notice that this shows that the resolvent $(\dot{A}- \text{Id} )^{-1}$ is well defined and continuous from $\dot{H}^1 \times L^2$ to $\dot{H}^2 \times H^1$. 
We have further:  \begin{lem}
  The injection $\dot{H}^2 \times {H}^1$ to $\dot{H}^1 \times L^2$ is compact.
  \end{lem}
  This follows from identifying $\dot{H}^s$ with the kernel of the linear form $u \mapsto \int_{M} u$. We also have:
  
  \begin{cor}\label{coroA.4}
  The operator $(\dot{A}- \text{Id} )^{-1} $ is compact on $ \dot{\mathcal{H}}$.
  \end{cor}
  On the other hand, it is very easy to show that for $(u_0, u_1) \in H^1\times L^2$, 
  $$\begin{pmatrix} \Pi & 0 \\ 0 & \text{Id} \end{pmatrix} e^{tA} = e^{t\dot{A}} \begin{pmatrix} \Pi & 0 \\ 0 & \text{Id} \end{pmatrix}, $$ and hence, the logarithmic decay  is equivalent to the logarithmic decay (in norm) of $e^{t\dot{A}}$ (and consequently, according to Theorem~\ref{thm} equivalent to resolvent estimates for $\dot{A}$). 
  The high frequency resolvent estimates in our new setting are  handled with the exact same proof as for Dirichlet boundary conditions (we did not use Poincaré's inequality in this regime) and consequently we omit the proof.   Let us focus on the low frequency regime and revisit our proof above in this new functional setting. We prove
  \begin{prop} Let $F \subset E \subset M$ and a damping $a$ satisfying (\ref{eq:hypothesis damping}). For any $\tau_*>0$ (we shall choose $\tau_*$ given by Proposition \ref{prop: high freq}), there exist constants $C >0$ such that for any $|\tau| < \tau_*$. and  every $h = (f,g) \in \mathcal{H}$ every $U$ such that $U = (\dot{A} - i \tau)^{-1} h$. 
\be
\| U \|_{\dot{H}^1(M) \times L^2(M)} \leq C \| (f,g ) \|_{\dot{H}^1(M) \times L^2(M)},
\label{eq:estimate low frequencies bis}
\ee 
\label{prop: low freq bis} 
\end{prop}

\begin{proof}
  We argue by contradiction. Suppose there exist sequences $(\tau_n), (U_n), (F_n)$ such that 
  $$ (\dot{A} - i \tau_n) U_n = F_n, \qquad  \|U_n\|_{\dot{\mathcal{H}} } > n \| F_n\|_{\dot{\mathcal{H}} }.$$ 
  Since $U_n \neq 0$, we can assume $\| U_n\|_{\dot{\mathcal{H}} }=1$. 
  Extracting subsequences (still indexed by $n$ for simplicity) we can also assume that $\tau_n \rightarrow \tau \in \mathbb{R}$ as $n\rightarrow  \infty $. We write 
  $$U_n= \begin{pmatrix} \dot{u}_n \\ v_n\end{pmatrix}, \qquad  F_n= \begin{pmatrix} \dot{f}_n \\ g_n\end{pmatrix},$$ and distinguish according to two cases. 
 
 \textbf{ Zero frequency case: $\tau =0$. }
  In this case, we have 
 $$ \dot{A} U_n = o(1)_{\dot{\mathcal{H}}} \Leftrightarrow \Pi v_n = o(1)_{\dot{H}^1}, \quad \Delta \dot{u}_n - a v_n = o(1)_{L^2}.$$
 We deduce that there exists $(c_n) \subset \mathbb{C}$ such that 
 $$ v_n - c_n = o(1)_{H^1}, \qquad \Delta u_n - a c_n = o(1)_{L^2}.$$
 But 
 $$ 0= \int_M \Delta u_n \ud x \Rightarrow c_n \int_M a \ud x = o(1) \Rightarrow c_n = o(1).$$
 As a consequence, we get $ v_n = o(1) _{L^2}$ and $\Delta u_n = o(1)_{L^2}$, which implies that $\dot{u} _n = o(1) _{\dot{H} ^1}$. This contradicts  $\| U_n\|_{{\dot{\mathcal{H}} }}=1$. As a result (\ref{eq:estimate low frequencies bis}) follows for $\tau = 0$. \par 
 
 \vspace{0.5em}
 
 \textbf{ Low (nonzero) frequency case: $\tau \in \mathbb{R}^*$.} 
 In this case, we have 
 $$ (\dot{A}-i \tau)  U_n = o(1)_{\dot{\mathcal{H}}} \Leftrightarrow \Pi v_n  - i \tau \dot{u} _n= o(1)_{\dot{H}^1}, \quad \Delta \dot{u}_n - ( i \tau +a)  v_n = o(1)_{L^2}.$$
 We deduce 
 $$ \Delta v_n - i\tau ( a+ i\tau) v_n = o(1)_{L^2} + \Delta (o(1) _{\dot{H}^1} ) = o(1)_{H^{-1}}.$$
 Since $(v_n)$ is bounded in $L^2$, from this equation, we deduce that $(\Delta v_n)$ is bounded in $H^{-1} $ and consequently $(v_n)$ is bounded in $H^1$. Extracting another subsequence, we can assume that $(v_n)$ converges in $L^2$ to $v$ which satisfies
 $$ \Delta v + \tau ^2 v - i\tau av =0, \qquad \textrm{in } M.$$
 Taking the imaginary part of the scalar product with $\overline{v}$ in $L^2$ gives (since $\tau \neq 0$) 
 $ \int_{M} a | v|^2 \ud x=0$, and consequently $av =0$ which implies that $v$ is an eigenfunction of the Laplace operator. But since the zero set of non trivial eigenfunctions has Lebesgue measure $0$ in $M$,  $av=0$ implies that $v=0$ (and consequently $v_n = o(1)_{L^2}$). Now, we have 
 $$ \Delta \dot{u}_n = (i\tau +a ) v_n + o(1)_{L^2} = o(1)_{L^2} \Rightarrow \dot{u}_n = o(1) _{\dot{H}^1},$$
 but this contradicts $\| U_n\|_{\dot{\mathcal{H}} }=1$ and (\ref{eq:estimate low frequencies bis}) follows also in this case. This ends the proof.

\end{proof}

Once we have established Proposition \ref{prop: low freq bis}, the proof of Theorem \ref{thm:decay} with Neumann boundary conditions follows the same lines of Section \ref{sec:proof Dirichlet} without significant modifications. This ends the proof of Theorem \ref{thm:decay}.

\section{Proof of Theorem \ref{energy decay Schrodinger}: Schr\"odinger equation}
\label{sec:Proof of Theorem schrodinger}
In order to prove Theorem \ref{energy decay Schrodinger} it is enough to prove the following resolvent estimates. 
\begin{prop}There exists $C>0$ such for any $\tau \in \mathbb{R}$, that the operator 
$$(\Delta - \tau + ia): D(A) \rightarrow L^2(M)$$
is invertible with bounded inverse 
$$ \| (\Delta - \tau + ia)^{-1} \|_{\mathcal{L}(L^2(M))} \leq C e^{c \sqrt{|\tau|}}.$$
\end{prop}

\subsection{Estimates when  $\tau < 0$} 
\label{sec:negative frequencies} 

In this case, we may use Proposition \ref{prop:Helmholtz} directly. We get the following result.

\begin{prop}
Assume that (\ref{eq:hypothesis damping}) holds for some $\alpha,\beta$ and let $f \in L^2(M;\CC)$ be given. Then, for any $\tau < 0 $, the resolvent $(\Delta - \tau + ia)^{-1}$ satisfies 
\begin{equation*}
\| (\Delta - \tau + ia)^{-1}f \|_{L^2(M)} \leq 2 (1 + C)^2 \left(  1  +   \beta  (1 + \frac{1}{\alpha} )^2   \right) e^{2 D\sqrt{|\tau|}}  \| f \|_{L^2},
\end{equation*} where $C,D$ are the constants in (\ref{eq:Helmholtz estimate}). 
\label{prop:Schrodinger negative}
\end{prop}

\begin{proof}
Recall that $\psi = (\Delta - \tau + ia)^{-1}f$ satisfies the Helmholtz equation (\ref{eq:Helmholtz Schrodinger}). 
Then, thanks to (\ref{eq:Helmholtz estimate}) there exist some constants $C,D>0$ independent of $\tau$ such that 
\begin{equation*}
\| \psi \|_{L^2(M)} \leq Ce^{D\sqrt{|\tau|}} \left(  \|i(f + a\psi) \|_{L^2(M)} + \| 1_{\omega} \psi \|_{L^2(M)} \right).
\end{equation*} On the other hand, (\ref{eq:hypothesis damping}) implies 
\begin{equation*}
\| 1_F \psi \|_{L^2(M)} \leq \frac{\sqrt{\beta}}{\alpha} \| \sqrt{a} \psi \|_{L^2(M)}
\end{equation*} and 
\begin{equation*}
\| a \psi \|_{L^2(M)} \leq \sqrt{\beta}  \| \sqrt{a} \psi \|_{L^2(M)}.
\end{equation*} As a result, we get 
\begin{equation}
\| \psi \|_{L^2(M)} \leq Ce^{D\sqrt{|\tau|}} \left( \| f \|_{L^2} + \left(  1 +  \frac{1}{\alpha}  \right) \sqrt{\beta}  \| \sqrt{a} \psi \|_{L^2(M)} \right).
\label{eq:thmSchrodinger main}
\end{equation}
Next, using the Helmholtz equation (\ref{eq:Helmholtz Schrodinger}) we obtain
\begin{equation*}
\int_M (\Delta \psi  -  \tau \psi ) \overline{\psi} \ud x  =  \int_M \left(  -if   - i a  \psi \right) \overline{\psi} \ud x
\end{equation*} and hence,
\begin{equation*}
\int_M |\nabla_x \psi|^2 \ud x + \tau \int_M  | \psi |^2 \ud x  =  i \int_M  f \overline{\psi} \ud x +  i \int_M a  |\psi|^2 \ud x.
\end{equation*} Now, taking the imaginary part and using Cauchy-Schwarz's and Young's inequalities we find 
\begin{align*}
 \left\| \sqrt{a} \psi  \right\|^2_{L^2(M)}  & = \Im \int_M f \overline{\psi} \ud x \\ 
 & \leq  \| f \|_{L^2(M)} \| \psi \|_{L^2(M)} \\
 & \leq \frac{1}{4 \epsilon} \| f \|_{L^2(M)}    +    \epsilon \| \psi \|_{L^2(M)},
\end{align*} for every $\epsilon > 0$. Injecting this in (\ref{eq:thmSchrodinger main}) yields 
\begin{equation*}
\| \psi \|_{L^2(M)} \leq Ce^{D\sqrt{|\tau|}} \left(  1  +   \sqrt{\beta} (1 + \frac{1}{\alpha} ) \frac{1}{2 \sqrt{\epsilon}}  \right) \| f \|_{L^2}  + Ce^{D\sqrt{|\tau|}}  \sqrt{\beta  \epsilon } (1 + \frac{1}{\alpha} )  \| \psi \|_{L^2}.
\end{equation*} Next, choosing 
\begin{equation*}
\epsilon = \frac{  e^{-2 D\sqrt{|\tau|}}  }{4   \left(  1 +  \frac{1}{\alpha}  \right)^2 \beta C^2  }
\end{equation*} we get 
\begin{equation*}
\| \psi \|_{L^2(M)} \leq 2 Ce^{D\sqrt{|\tau|}} \left(  1  +   \beta  (1 + \frac{1}{\alpha} )^2 C e^{D\sqrt{|\tau|}}  \right) \| f \|_{L^2}.
\end{equation*} As a consequence, we get the exponential growth estimate
\begin{equation*}
\| \psi \|_{L^2(M)} \leq 2 (1 + C)^2 \left(  1  +   \beta  (1 + \frac{1}{\alpha} )^2   \right) e^{2 D\sqrt{|\tau|}}  \| f \|_{L^2},
\end{equation*} for any $\tau > 0 $ given.

\end{proof}

\subsection{Estimates when $0 \leq \tau $} 
\label{sec:positive frequencies} 
In this section we shall just rely on the following Poincaré-type inequality:
\begin{prop}
Assume that $a\geq 0$ and $\int_M a(x) dx >0$. Then there exists $C_P = C_P(a)>0$ such that for all $u\in H^1(M)$,
\begin{equation}
C_P \int_M ( |\nabla_x u |^2(x)  + a(x) |u|^2(x) ) dx \geq  \| u \| _{H^1(M)}^2.
\label{eq:Poincare}
\end{equation}
\end{prop}

\begin{proof}

We follow a standard proof and argue by contradiction.
Otherwise, there would exist a sequence $(u_n) \in H^1(M)$ (that we can assume of norm $1$ in $H^1$) such that
$$ \int_M ( |\nabla_x u_n |^2(x)  + a(x) |u_n|^2(x) ) dx \leq \frac 1 n \| u_n\|^2_{H^1(M)}.$$
Since $(u_n)$ is bounded in $H^1(M)$ and $M$ is compact, by Rellich's compactness theorem there exists $u\in H^1$ such that we can extract a subsequence (still denoted by $(u_n)$) such that 
$$ \|u_n -u \|_{L^2} \rightarrow 0, \qquad \textrm{as } n \rightarrow + \infty.$$
Moreover, as
$$\| \nabla_x u_n\|_{L^2} \rightarrow 0, \qquad \|\sqrt{a} u_n \|_{L^2(M)} \rightarrow 0,$$
 we deduce that $\nabla_x u =0 $ and thus $u$ must be constant in $M$. But
 $$|u|^2 \int_M a(x) dx = \int_M a(x) |u|^2 (x) dx =\lim_{n\rightarrow + \infty}\int_M a(x) |u_n|^2 (x) dx,$$
 which implies that $u=0$. This gives a contradiction with the fact that 
 $$ \| u_n\|_{H^1}=1, \quad \|\nabla_x u_n \|_{L^2} \rightarrow 0, \quad \| u_n -u\|_{L^2} \rightarrow 0 \qquad \Rightarrow \qquad \| u\| _{L^2} =1.
 $$  Hence, (\ref{eq:Poincare}) follows for some positive constant $C_P$.

\end{proof}

\begin{prop}
Assume that (\ref{eq:hypothesis damping}) holds for some $\alpha,\beta$ and let $f \in L^2(M;\CC)$ be given. Then, there exists $C>0$ such that for  $\tau \geq 0 $,and any $f\in L^2(M)$, we have  
\begin{equation*}
\| ( \Delta - \tau + i a(x)) ^{-1} f \|_{L^2(M)} \leq  2C_P  \| f \|_{L^2(M)},
\end{equation*} where $C_P$ is  the Poincaré's constant above. 
\label{prop:Schrodinger positive}
\end{prop}

\begin{proof}

For $f \in L^2(M)$ given, let $\psi = (A - \tau + ia)^{-1}f$. Recalling that $\psi$ satisfies the Helmholtz equation (\ref{eq:Helmholtz Schrodinger}), 
after multiplying by $\overline{\psi}$ and integrating by parts, we get
\begin{equation}\label{ert}
-\int_M \vert \nabla \psi \vert^2 \ud x - \tau \int_M \vert \psi \vert^2 \ud x + i  \int_M a(x) \vert \psi \vert^2 \ud x  = - i \int_M f \overline{\psi} \ud x.
\end{equation}
The modulus of the left hand side in~\eqref{ert} is larger that 
$$ \frac 1 2 \int_M (\vert \nabla \psi \vert^2 (x) + a(x) |\psi|^2(x) ) dx.$$  
 Using Poincaré's inequality (\ref{eq:Poincare}) in the left and Cauchy-Schwarz in the right we get
$$
 \|\psi\|^2 _{H^1} \leq 2C_P \| f\|_{L^2} \| \psi\|_{L^2} \Rightarrow \| \psi\| _{H^1} \leq 2C_P \| f\|_{L^2}.
 $$

\end{proof}

\subsection{Conclusion of the proof of Theorem \ref{energy decay Schrodinger}}

\begin{proof}[Proof of Theorem \ref{energy decay Schrodinger}]
Combining Proposition \ref{prop:Schrodinger positive} and Proposition \ref{prop:Schrodinger negative}, we find that the estimate 
\begin{equation*}
\| R(\tau)f \|_{L^2(M)} \leq  C_0  e^{2 D\sqrt{|\tau|}}  \| f \|_{L^2}
\end{equation*} holds for every $\tau \in \RR$, every $f \in L^2(M)$ and the constant
\begin{equation*}
C_0 := \max  \left\{  2 (1 + C)^2 \left(  1  +   \beta  (1 + \frac{1}{\alpha} )^2   \right),    2C_P   \right\}.
\end{equation*} As a consequence the function $M$ defined by (\ref{eq:function M}) satisfies in this case the growth
\begin{equation*}
M(\mu)  \leq C_1 e^{2 D\sqrt{|\mu|}},   
\end{equation*} for every $\mu \in \RR$. Hence, Theorem \ref{thm:Burq} yields that for any $k \in \NN$, 
\begin{equation}
\left\|  U(t) (Id - A)^{-k}  \right\|_{\mathcal{L}(L^2)} \leq \frac{C_k}{\log(2 + t)^{2k}}, \qquad \forall t \geq 0.
\label{eq:decay Schrodinger en norme}
\end{equation} where we have used the notation of (\ref{eq:operator Schrodinger}) and (\ref{eq:semigrooup Schrodinger}). Next, let $k \in \NN$ be fixed and let $\psi_0 \in H^{2k}(M)$ Then, $\Phi_0 := (Id - A)^{k} \psi_0 \in L^2(M)$ and inequality (\ref{eq:decay Schrodinger en norme}) implies 
\begin{align*}
\mathcal{E}_S(t,\psi_0) & = \left\|  U(t) \psi_0  \right\|^2_{L^2} \\
& = \left\|  U(t) (Id - A)^{-k} \Phi_0  \right\|^2_{L^2} \\
& \leq \frac{C_k}{\log(2 + t)^{4k}}  \left\| \Phi_0 \right\|^2_{L^2} \\
& \leq \frac{C_k}{\log(2 + t)^{4k}}  \left\| (Id - A)^{k} \psi_0 \right\|^2_{L^2},
\end{align*} for every $t \geq 0$. 
 The proof is completed by noticing that for $k=1$ the domain of $A$ is $H^2(M)\cap H^1_0 (M) $ with Dirichlet boundary conditions  (resp. $H^2(M)$ with Neumann boundary conditions or if $ \partial M = \emptyset$)   and  
$$\| (Id - A) \psi_0 \|^2_{L^2} \leq C \|  \psi_0 \|^2_{H^2 }.$$
\end{proof}

\section*{Acknoledgments} 
N. Burq is  partially supported by the grant   "ISDEEC'' ANR-16-CE40-0013 and Institut Universitaire de France.
I. Moyano is very grateful to AIMS Ghana, where part of this work was done during a visit in August 2021.

\bibliographystyle{plain}                            
{}

\end{document}